\documentclass[11pt]{amsart}
 \usepackage[a4paper, margin=0.8in]{geometry}
\usepackage{hyperref}
\usepackage{color}

\def\IN{{\Bbb N}}
\def\IR{{\Bbb R}}
\def\U{{\mathcal U}}
\def\sji{sj^\infty}
\def\SJI{SJ^\infty(X)}
\def\e{\varepsilon}
\def\pto{\mapsto}
\def\t{\times}
\def\td{\tilde d}
\def\xyt{[x,y;t]}
\def\bs{\backslash}

\def\d{\delta}
\def\s{\sigma}
\def\card{\operatorname{card}}

\def\diam{\operatorname{diam}}
\def\supp{\operatorname{supp}}
\def\pr{\,\text{pr}}

\newtheorem{theorem}{Theorem}[section]
\newtheorem{lemma}[theorem]{Lemma}

\theoremstyle{definition}
\newtheorem{remark}[theorem]{Remark}

\begin{document}

\title{On regular operators extending (pseudo)metrics}
\author{Taras Banakh}
\subjclass{54C20, 54C35, 54E20, 54E35}
\address{Department of Mechanics and Mathematics, Lviv University, Universytetska 1, Lviv, 290602, Ukraine}

\begin{abstract} It is proved that for every stratifiable  space $Y$ and a closed
subset $X\subset Y$ there exists a regular (i.e. linear positive with unit norm) extension 
operator $T:C(X\times X)\to C(Y\times Y)$ 
preserving the class of (pseudo)metrics.
This operator is continuous with respect 
to the pointwise  as well as to the compact-open topologies on the linear 
lattices of continuous functions $C(X\t X)$ and $C(Y\t Y)$.
If moreover the space Y is metrizable 
then the operator $T$ preserves the class of admissible metrics. The equivariant analog of the above statement is proved as well.
\end{abstract}

\maketitle

Let $X$ be a closed subset in a metrizable compactum $Y$. Let's recall two
classical results due respectively to Hausdorff \cite{11} and Dugundji \cite{8} (see also \cite{1}): 
(i) every metric on $X$ can be extended onto $Y$ and (ii)
there exists a bounded linear operator $T:C(X)\to C(Y)$ (here $C(Z)$
is the Banach space of all  continuous functions on $Z$) such that
$T(f)|X=f$ for every $f\in C(X)$ (such an operator is called an extension operator). 

In \cite{3}, C.~Bessaga 
constructed a bounded linear operator
extending metrics 
for every metric pair $(Y,X)$, where $X$ is a
neighborhood retract in $Y$,
and asked whether every compact pair admits such an  extension
operator.

Before answering (in positive) this question, recall at first some facts
concerning extension operators. There are two different fundamental
approaches to the proof of the above-mentioned Dugundji theorem (and
respectively two directions of its generalization). The first one is due
to Dugundji and exploits the technique of normal covers. The best known
(and rather the best possible) generalization of Dugundji Theorem in this
direction belongs to Carlos J.R.~Borges who proved in \cite{6} that for every
stratifiable space $Y$ and its closed subset $X$ there is a regular (i.e.
linear positive sending the unit function on $X$ to the unit function on
$Y$) extension operator $T:\IR^X\to \IR^Y$ which preserves continuous
functions. Moreover, the operator $T$ is continuous whenever both $\IR^X$
and $\IR^Y$ have the compact-open topology, the topology of point-wise
convergence, or the topology of uniform convergence.

The second approach to the proof of Dugundji Theorem is based on the
Milutin epimorphism and averaging operators, and allows to generalize
Dugundji Theorem to the class of non-metrizable compacta. Namely, it was
proved in \cite{12} that for every compact pair $(Y,X)$, $X\subset Y$, where $X$ 
is an $AE(0)$-compactum, there is a regular extension operator $T:C(X)\to
C(Y)$. Moreover, if a compactum $X$ admits a regular
extension operator for every compact $Y\supset X$ (such compacta $X$ are
called Dugundji \cite{13}) then $X$ is necessarily $AE(0)$.

Now we have known the classes of spaces for which the Dugundji Theorem holds,
so we have known  for which pairs one can expect to construct regular
operators extending (pseudo)metrics.

In \cite{2}, slightly modifying the classical construction of the
averaging operator for the Milutin epimorphism, the author has proved that 
for every
Dugundji compactum $X$ and every compact $Y\supset X$ there is a regular
operator $T:C(X\times X)\to C(Y\times Y)$ extending pseudometrics from $X$
onto $Y$.

The aim of this paper is to obtain the corresponding result for the class
of all stratifiable spaces. Namely, we will prove that for every
stratifiable space $Y$ and its closed subspace $X$, $|X|>1$, there is a
regular extension operator $T:\IR^{X\times X}\to \IR^{Y\times Y}$ which
preserves the classes of bounded functions, continuous functions, and the
class of (pseudo)metrics. If moreover, the space $Y$ is metrizable, we can
require for $T$ to extend admissible metrics on $X$ to admissible metrics on
$Y$. The operator $T$ is continuous whenever both $\IR^{X\times X}$ and
$\IR^{Y\times Y}$ have the compact-open topology, the topology of
point-wise convergence, or the topology of uniform convergence.

The idea of the
proof consists in finding a corresponding functor $F$ which for every
space $X$ admits a ``nice" operator extending (pseudo)metrics from $X$
onto $F(X)\supset X$ and such that for every stratifiable space $Y$ containing $X$ 
as a closed subset there is a map $f:Y\to F(X)$ which extends the identity
embedding $X\to F(X)$. This idea has been exploited by M.M.~Zarichnyi, who, using the
construction of extending metrics from a compactum onto its probability
measures space, has constructed a bounded (but non-linear) operator
extending metrics for each compact pair.

The author would like to express his thanks to R.Cauty, who turned
the author's attention to the class of stratifiable spaces, and to M.Zarichnyi and
C.~Bessaga for valuable and stimulating discussions on the subject of the
paper. Special thanks are expressed to the referee whose numerous critical 
remarks have allowed the author to improve essentially the presentation of 
results.

\section{The squeezed join construction and the great magic formula}

By analogy with the squeezed cone construction \cite{4}, we define the
construction of squeezed join. Let $X$ be a topological space. Consider
the set $X\times X\times [0,1]$. For each $x,x',y,y'\in X$ and
$t,t'\in [0,1]$ let $(x,y,0)\sim (x,y',0)$, $(x,y,1)\sim (x',y,1)$ and
$(x,x,t)\sim (x,x,t')$. Obviously that $\sim$ generates the equivalence
relation on $X\times X\times [0,1]$. Consider the quotient-set $SJ(X)=X\times
X\times [0,1]/\sim$ and denote by $[x,y;t]$ the equivalence class of the
element $(x,y,t)\in X\times X\times [0,1]$. Note that the set $X$ can be
embedded into $SJ(X)$ by the formula $x\pto [x,x;0]=\{x\}\times
X\times\{0\}\cup X\times\{x\}\t\{1\}\cup\{x\}\t\{x\}\t[0,1]$. If
$[x,y;t]\notin X\subset SJ(X)$ then $[x,y;t]=\{(x,y,t)\}$ consists of the only
element. Now we define topology on $SJ(X)$. Fix $[x,y;t]\in SJ(X)$. If
$\xyt\notin X\subset SJ(X)$ then the sets of the type
$[U_x,U_y;U_t]=\{[x',y';t']\in SJ(X)\mid (x',y',t')\in U_x\t U_y\t U_t\}$,
where $U_x,\; U_y,\; U_t$ are respectively neighborhoods of $x,y,t$, form
the base of neighborhoods of $\xyt$. If $\xyt=x\in X$ then the base of
neighborhoods of $\xyt$ consists of the sets $[U,\e]=\{[x',y';t']\in
SJ(X)\mid (x',y',t')\in U\t U\t[0,1]\cup U\t X\t[0,\e)\cup X\t
U\t(1-\e,1]\}$ where $\e>0$ and $U$ runs over neighborhoods of
$x=[x,y;t]\in X$. The topological space $SJ(X)$ is called {\it the
squeezed join} of $X$. Remark that  the inclusion $X\hookrightarrow SJ(X)$
is a topological embedding, so further we consider the space $X$ as a
subset in $SJ(X)$.

Notice that the squeezed join $SJ(X)$ can be considered as the union of
segments connecting points of $X$. Namely, $[x,y]=\{[x,y;t]:t\in [0,1]\}$
is the segment connecting the points $x\in X$ and $y\in Y$. Remark that
$[x,y]\ne[y,x]$ for $x\ne y$; moreover $[x,y]\cap [y,x]=\{x,y\}$. By the
definition of the topology on $SJ(X)$ the natural map $X\times X\times
[0,1]\to SJ(X)$ acting by $(x,y,t)\mapsto [x,y;t]$ is continuous.

Every continuous map $f:X\to Y$ generates a continuous map 
$SJ(f):SJ(X)\to SJ(Y)$ acting as $SJ(f)([x,y;t])=[f(x),f(y);t]$ for
$[x,y;t]\in SJ(X)$. Hence the squeezed join construction determines a
functor $SJ$ in the category ${\mathcal T}op$ of topological spaces and their
continuous maps. 

For an  element $a\in SJ(X)$ we define its support $\supp(a)\subset X$ as 
follows: if $a\in X\subset SJ(X)$ then $\supp(a)=\{a\}$; if $a=[x,y;t]\in
SJ(X)\bs X$ we put $\supp(a)=\{x,y\}$.

Assume that $u=(x,y,t)$ and $v=(x',y',t')$ are points of $X\t X\t [0,1]$
and $p:X\t X\to \IR$ is a real-valued function on $X\t X$. The following
formula
$$
sj(p)(u,v)=\min\{1-t,1-t'\}p(x,x')+\max\{0,t'-t\}p(x,y')
+\max\{0,t-t'\}p(y,x')+\min\{t,t'\}p(y,y')
$$
will be referred to as {\color{magenta}\it great magic formula}.

One can easily verify that $sj(p)(u,v)=sj(p)(u',v')$, whenever $u\sim u'$
and $v\sim v'$. Hence, it is legal to write $sj(p):SJ(X)\t SJ(X)\to \IR$.
Remark that for every $x,y\in X\subset SJ(X)$ \ $sj(p)(x,y)=p(x,y)$, i.e.
the map $sj(p)$ extends the map $p$..
\vskip5pt

The great magic formula considered as a real function of the
variable $(t,t')\in [0,1]\times [0,1]$ with fixed $x,x',y,y'\in X$ is just
a map from the square $[0,1]\times [0,1]$, which is linear on each of the
triangles $\Delta_1=\{(t,t')\in [0,1]^2\mid t\le t'\}$, 
$\Delta_2=\{(t,t')\in [0,1]^2\mid t\ge t'\}$, and takes the values
$p(x,x')$, $p(x,y')$, $p(y,x')$ and  $p(y,y')$ at the vertises $(0,0)$,
$(0,1)$, $(1,0)$ and $(1,1)$ respectively.
\vskip5pt

\begin{remark}[partly justifying the choise of terms] The squeezed cone
 $sc_o(X)$ of a pointed space $(X,o)$ (see \cite{3,4}) is naturally embeddable into the
squeezed join of $X$: $(x,t)\pto [o,x;t]$, $(x,t)\in sc_o(X)$.
Moreover
the magic formula from \cite{4} coincides up to one summand with  the restriction 
of the great magic formula onto the squeezed cone $sc_o(X)\subset SJ(X)$.
\end{remark}

For a map $p:X\t X\to \IR$ let $\|p\|=\sup\{|p(x,y)| : (x,y)\in X\t X\}\in
[0,\infty]$.

\begin{lemma}\label{l:1.1} For every map $p:X\t X\to \IR$ $\|sj(p)\|=\|p\|$.
Furthermore, if $p\equiv 1$ then $sj(p)\equiv 1$; if $p\ge 0$ then $sj(p)\ge 0$.
For every maps $p,p':X\times X\to \IR$ and $a,b\in SJ(X)$
\begin{itemize}
\item if
$(p-p')|\supp(a)\times \supp(b)\equiv 0$ then $sj(p)(a,b)=sj(p')(a,b)$;
\item if
$|p(x,x')|<1$ for every $(x,x')\in \supp(a)\times \supp (b)$ then
$|sj(p)(a,b)|<1$.
\end{itemize}
\end{lemma}

\begin{proof} The first and the second statements follow from the equality
$$\min\{1-t,1-t'\}+\max\{0,t'-t\}+\max\{0,t-t'\}+\min\{t,t'\}=1$$ 
holding for each $t,t'\in[0,1]$. The rest statements of Lemma easily 
follows from the great magic formula. 
\end{proof}

\begin{lemma}\label{l:1.2} For every bounded continuous map $p:X\t X\to\IR$
the map \newline 
$sj(p):SJ(X)\t SJ(X)\to\IR$ is continuous.
\end{lemma}

\begin{proof} Let $p:X\t X\to \IR$ be a bounded continuous map. Without
loss of generality, $\|p\|=1$. It follows from the great magic formula and
the definition of the topology on $SJ(X)$ that the map $sj(p)$ is
continuous on the set $(SJ(X)\bs X)\t (SJ(X)\bs X)$. To verify continuity
of $sj(p)$ on the set $(X\t SJ(X))\cup(SJ(X)\t X)$ fix $\xyt\in SJ(X)$,
$z\in X$ and $\e>0$.

Assume at first, $\xyt\in SJ(X)\bs X$. Withot loss of generality, 
$\e<\min\{t,1-t\}$. Remark that $sj(p)([x,y;t],z)= (1-t)p(x,z)+tp(y,z)$.
Since the function $p$ is continuous,
there exist neighborhoods $U_x,U_y,U_z\subset X$ of the points $x,y,z$
such that $p(U_x\t U_z)\subset (p(x,z)-\frac\e6,p(x,z)+\frac\e6)$ and
$p(U_y\t U_z)\subset (p(y,z)-\frac\e6,p(y,z)+\frac\e6)$. Let
$U_t=(t-\frac\e6,t+\frac\e6)$. By the definition, the sets
$[U_x,U_y;U_t]$ and $[U_z;\frac\e6]$ are neighborhoods of $\xyt$
and $z$ respectively. 
Let us show that for any $a\in [U_x,U_y;U_t]$ and $b\in [U_z,\frac\e6]$ \
$|sj(p)(a,b)-sj(p)([x,y;t],z)|<\e$. Fix $a=[x',y';t']\in [U_x,U_y;U_t]$
and $b=[x'',y'';t'']\in [U_z,\frac\e6]$. There are three cases:
$t''<\frac\e6$, $t''>1-\frac\e6$ or $x'',y''\in U_z$. Assuming the first one, notice
 that $t'>t''$ and $x''\in U_z$. Then
$|sj(p)(a,b)-sj(p)([x,y;t],z)|=|(1-t')p(x',x'')+(t'-t'')p(y',x'')+
t''p(y',y'')-(1-t)p(x,z)-tp(y,z)|=|(1-t)(p(x',x'')-p(x,z))+(t-t')p(x',x'')+
t(p(y',x'')-p(y,z))+(t'-t)p(y',x'')+t''(p(y',y'')-p(y',x''))|\le
(1-t)|p(x',x'')-p(x,z)|+|t-t'|\,\|p\|+t|p(y',x'')-p(y,z)|+(t'-t)\|p\|+
t''\|p\|\le \frac\e6+\frac\e6+\frac\e6+\frac\e6+\frac\e6=\frac56\e<\e$.
The case $t''>1-\frac\e6$ can be treated by analogy.

Now assume that $x'',y''\in U_z$. There are two possibilities: $t'\le t''$ or 
$t''\le t'$. Consider the first one. Then $|sj(p)(a,b)-sj(p)([x,y;t],z)|=
|(1-t'')p(x',x'')+(t''-t')p(x',y'')+t'p(y',y'')-(1-t)p(x,z)-tp(y,z)|=
|(1-t'')(p(x',x'')-p(x,z))+(t''-t')(p(x',y'')-p(x,z))+t'(p(y',y'')-p(y,z))+
(t'-t)p(x,z)+(t-t')p(y,z)|\le |p(x',x'')-p(x,z)|+|p(x',y'')-p(x,z)|+
|p(y',y'')-p(y,z)|+|t'-t|(|p(x,z)|+|p(y,z)|)<\frac\e6+\frac\e6+\frac\e6+2\frac\e6<\e$.
The case $t''\le t'$ can be considered by analogy.

 Therefore the
function $sj(p)$ is continuous on the set $(SJ(X)\times SJ(X))\bs (X\times
X)$.
The following Lemma (which will be essentially used later) completes the
proof.  
\end{proof}

\begin{lemma}\label{l:1.3} Let $a,b\in X$, $\e>0$, $p:X\t X\to \IR$ be a map and
$U_a,U_b\subset X$ be neighborhoods of $a,b$ such that for every $x\in
U_a$, $y\in U_b$ $|p(x,y)-p(a,b)|<\e$. Then for every $\delta>0$ and
$x\in[U_a,\d]$, $y\in[U_b,\d]$ we have $|sj(p)(x,y)-p(a,b)|<\e+3\d\|p\|$.
\end{lemma}

The proof does not contain any principal difficulties and therefore is
omitted.

\begin{lemma}\label{l:1.4} If $p:X\t X\to\IR$ is a (pseudo)metric on $X$
then $sj(p)$ is a (pseudo)metric for $SJ(X)$.
\end{lemma}

\begin{proof} Verification of all the axioms of a (pseudo)metric excepting
the triangle inequality is rather trivial.

Let $\xyt,[x',y';t'],[x'',y'';t'']\in SJ(X)$. We shall prove that
$$sj(p)(\xyt,[x',y';t'])\le sj(p)(\xyt,[x'',y'';t''])+
sj(p)([x'',y'';t''],[x',y';t']).$$

Since the map $sj(p)$ is symmetric, without loss of generality, $t\le t'$.

Assume at first $t\le t'\le t''$. Then
$$
\begin{aligned}
&sj(p)(\xyt,[x',y';t'])=(1-t')p(x,x')+(t'-t)p(x,y')+tp(y,y')\le\\
&\le(1-t'')(p(x,x'')+p(x'',x'))+(t''-t')p(x,x')+(t'-t)(p(x,y'')+p(y'',y'))+\\
&+tp(y,y'')+ tp(y'',y')\le (1-t'')p(x,x'')+(1-t'')p(x'',x')+(t''-t')p(x,x')+\\
&+(t''-t)p(x,y'')+(t'-t'')p(x,y'')+(t'-t)p(y'',y')+tp(y,y'')+tp(y',y'')\le\\
&\le(1-t'')p(x,x'')+(t''-t)p(x,y'')+tp(y,y'')+(1-t'')p(x',x'')+\\
&+(t''-t')(p(x,x')-p(x,y''))+t'p(y',y'')\le\\
&\le sj(p)(\xyt,[x'',y'';t''])+(1-t'')p(x',x'')+(t''-t')p(x',y'')+t'p(y',y'')=\\
&=sj(p)(\xyt,[x'',y'';t''])+sj(p)([x'',y'';t''],[x',y';t']).
\end{aligned}
$$

Now assume that $t\le t''\le t'$.
Then
$$
\begin{aligned}
&sj(p)([x,y;t],[x',y',t'])=(1-t')p(x,x')+(t'-t)p(x,y')+tp(y,y')\\
&\le
(1-t')(p(x,x'')+p(x'',x'))+(t'-t'')p(x,y')+(t''-t)p(x,y')+
t(p(y,y'')+p(y'',y'))\\
&\le (1-t')(p(x,x'')+p(x'',x'))+(t'-t'')(p(x,x'')+p(x'',y'))+(t''-t)(p(x,y'')+
p(y'',y'))\\
&+t(p(y,y'')+p(y'',y'))=(1-t'')p(x,x'')+(t''-t)p(x,y'')+tp(y,y'')+
(1-t')p(x',x'')\\
&+(t'-t'')p(x'',y')+t''p(y,y'')
=sj(p)([x,y;t],[x'',y'';t''])+sj(p)([x'',y'';t''],[x',y';t']).
\end{aligned}
$$
Let finally $t''\le t\le t'$. Then
$$
\begin{aligned}
&sj(p)([x,y;t],[x',y';t'])=(1-t')p(x,x')+(t'-t)p(x,y')+tp(y,y')\\
&\le(1-t')(p(x,x'')+p(x'',x'))+(t'-t)(p(x,x'')+p(x'',y'))+t''p(y,y')+(t-t'')p(y,y')\\
&\le(1-t)p(x,x'')+(1-t')p(x',x'')+(t'-t)p(x'',y')+t''(p(y,y'')+p(y'',y'))\\
&+(t-t'')(p(y,x'')+p(x'',y'))
=(1-t)p(x,x'')+(t-t'')p(x'',y)+t''p(y,y'')\\
&+(1-t')p(x',x'')+(t'-t'')p(x'',y')+t''p(y'',y')\\
&=sj(p)([x,y;t],[x'',y'';t''])+sj(p)([x'',y'';t''],[x',y';t']).
\end{aligned}
$$
Therefore $sj(p)$ is a pseudometric on $SJ(X)$.  
\end{proof}

For a metric space $(X,d)$, by $O_d(x,\e)$ the
open $\e$-ball $\{x'\in X\mid d(x,x')<\e\}$ of the point $x\in X$ is
denoted.

Since we have proved that $sj(d)$ is a pseudometric whenever $p$ is,
remark one simple fact. Namely, it follows from the great magic formula
that for $x,y,z\in X$ and $t\in [0,1]$ \
$sj(p)([x,y;t],z)=(1-t)p(x,z)+tp(y,z)\le\max\{p(x,z),p(y,z)\}$, i.e. every
$\e$-ball $O_{sj(p)}(z,\e)$, $z\in X$, contains together with points
$x,y\in O_{sj(p)}(z,\e)\cap X$ also the segment $[x,y]$.

A metric $d$ on a topological space $X$ is called {\it
dominating}, if for every point $x\in X$ and its open neighborhood
$U\subset X$ there exists $\e>0$ such that $O_d(x,\e)\subset U$. A metric
which is both continuous and dominating is called {\it admissible}.

\begin{lemma}\label{l:1.5} If $d$ is a dominating metric on $X$ then $sj(d)$ is
a dominating metric on $SJ(X)$.
\end{lemma}

\begin{proof} Lemma~\ref{l:1.4} implies that $sj(d)$ is a metric
on $SJ(X)$. To show that the metric $sj(d)$ is dominating, fix $\xyt\in
SJ(X)$ and its neighborhood $U\subset SJ(X)$.

Consider at first the case $\xyt\in SJ(X)\bs X$. By the definition, there
exist neighborhoods $U_x,U_y\subset X$ and $U_t\subset[0,1]$ of $x,y,t$
such that $[U_x,U_y;U_t]\subset U$. Let $\e<\min\{t/2, (1-t)/2\}$ 
be a positive real such
that $O_d(x,\e)\subset U_x$, $O_d(y,\e)\subset U_y$ and 
$(t-\e,t+\e)\subset
U_t$. Let  $[x',y';t']\in SJ(X)$ satisfy the inequality
$ sj(d)([x',y';t'],\xyt)<
\delta=\min\{\e^2/2,\e\,d(x,y)\}$. We shal show 
that $[x',y';t']\in[U_x,U_y;U_t]\subset U$.
Assuming the converse we will obtain that either $t'\notin (t-\e,t+\e)$, or
$t'\in(t-\e,t+\e)$ and $(x',y')\notin U_x\t U_y$. Suppose at first that
$t'>t+\e$. Then $sj(d)([x',y';t'],[x,y;t])=
(1-t')d(x,x')+(t'-t)d(x,y')+td(y,y')\ge \e(d(x,y')+d(y',y))\ge
\e\,d(x,y)\ge\delta$. This is a contradiction. The case $t'<t-\e$ can be
considered quite analogously. Now assume that $t'\in (t-\e,t+\e)$ but
$x'\notin U_x$, or $y'\notin U_y$. Consider at first the case $x'\notin
U_x$. Then $d(x,x')>\e$ and $sj(d)([x',y';t'],[x,y;t])
\ge \min\{1-t,1-t'\}d(x,x')
\ge(1-(t+\e))\cdot \e\ge \e^2/2\ge\delta$.
This is a contradiction. By analogy, one can show that the assumption 
$y'\notin U_y$ also leads to a contradiction.

If $\xyt=x\in X\subset SJ(X)$ then, by the definition, there exist a
neighborhood $U_x\subset X$ of $x$ and $\e>0$ such that $[U_x;\e]\subset
U$. Without loss of generality $O_d(x,\e)\subset U_x$. 
Assume $sj(d)([x',y';t'],x)<\e^2$ for
$[x',y';t']\in SJ(X)$. It is easily seen that for $\e\le t'\le1-\e$,
inequality $sj(d)([x',y';t'],x)=(1-t')d(x,x')+t'd(x,y')<\e^2$ implies
$d(x,x')<\e$ and $d(x,y')<\e$. This means $[x',y';t']\in[U_x;\e]$ and
$O_{sj(d)}(x,\e^2)\subset U$.
  \end{proof}

Lemmas~\ref{l:1.2} and \ref{l:1.5} immediately imply

\begin{lemma}\label{l:1.6} If $d$ is an admissible bounded metric on $X$
then $sj(d)$ is an admissible bounded metric for $SJ(X)$.
\end{lemma}

Recall that a pseudometric $p$ on $X$ is defined to be {\it totally
bounded} provided for every $\e>0$ there exists a finite $\e$-net (i.e.~a
finite set $A\subset X$ such that for every $x\in X$
$p(x,A)=\min\{p(x,a)\mid a\in A\}<\e$).

\begin{lemma}\label{l:1.7} If $p$ is a totally bounded pseudometric on $X$
then $sj(p)$ is a totally bounded pseudometric on $SJ(X)$.
\end{lemma}

\begin{proof} Without loss of generality $0<\|p\|<\infty$. 
Fix $\e>0$. Let $A$ be
a finite $\frac\e4$-net on $X$ and $B\subset[0,1]$ be a finite
$\frac\e{4\|p\|}$-net for the interval $[0,1]$ equipped with the standard
metric. It follows from the great magic formula that the set $C=\{\xyt\in SJ(X)\mid x,y\in A,\;
t\in B\}$ is a finite $\e$-net for $(SJ(X),sj(p))$.  
\end{proof}

\begin{lemma}\label{l:1.8} If $p$ is a complete bounded pseudometric on $X$ then $sj(p)$
is a complete bounded pseudometric on $SJ(X)$.
\end{lemma}

\begin{proof} Let $p$ be a complete bounded pseudometric on $X$. Without
loss of generality, $\|p\|=1$. By Lemma~\ref{l:1.4}, $sj(p)$ is a bounded
pseudometric on $SJ(X)$. Let $\{u_n\}_{n=1}^\infty\subset SJ(X)$ be a
$sj(p)$-Cauchy sequence in $SJ(X)$.

Consider firstly the case
$\varliminf\limits_{n\to\infty}sj(p)(u_n,X)=\lim\limits_{n\to\infty}\inf\{sj(p)(u_k,X)\mid
k\ge n\}=0$. By going to a subsequence, we can assume that
$\lim_{n\to\infty}sj(p)(u_n,X)=0$. Let $\{x_n\}_{n=1}^\infty\subset X$ be
a sequence in $X$ such that $\lim_{n\to\infty}sj(p)(u_n,x_n)=0$. It is
easily seen that the sequence $\{x_n\}_{n=1}^\infty\subset X$ is
$sj(p)$-Cauchy. Since the pseudometric $sj(p)$ extends the complete
pseudometric $p$ on $X$, there exists a $p$-limit point
$x=\lim_{n\to\infty}x_n\in X$, which is also a $sj(p)$-limit point of the sequence
$\{u_n\}_{n=1}^\infty$.

Now assume that $sj(p)(u_n,X)>\delta$ for some $\delta>0$. Remark that
each equivalence class $u_n$ consists of the only element to say
$(x_n,y_n,t_n)$. The inequality $sj(p)(u_n,X)>\delta$ implies
$p(x_n,y_n)>\delta$ and $\delta<t_n<1-\delta$ for each
$n\in\IN$. Then
$sj(p)(u_n,u_m)=\min\{1-t_n,1-t_m\}p(x_n,x_m)+\max\{0,t_m-t_n\}p(x_n,y_m)+\max\{0,t_n-t_m\}p(x_m,y_n)+\min\{t_n,t_m\}p(y_n,y_m)<\e$
implies $p(x_n,x_m)<\e/\delta$, $p(y_n,y_m)<\e/\delta$ and
$|t_n-t_m|<\frac{\e\cdot\delta}{\delta^2-\e}$
(the last inequality holds because $p(x_n,y_m)\ge p(x_n,y_n)-p(y_n,y_m)>
\delta-\frac\e\delta$ and $p(x_m,y_n)>\delta-\frac\e\delta$~).
 This yields that the sequences
$\{x_n\}_{n=1}^\infty\subset X$, $\{y_n\}_{n=1}^\infty\subset X$ and
$\{t_n\}_{n=1}^\infty\subset [0,1]$ are Cauchy. Since the space $X$ is
$p$-complete, there exist limit points $x=\lim_{n\to\infty}x_n\in X$,
$y_n=\lim_{n\to\infty}y_n\in X$ and $t=\lim_{n\to\infty}t_n\in [0,1]$. It
is easily seen that $\lim_{n\to\infty}sj(p)(u_n,u)=0$, where $u=[x,y;t]$.
Therefore, the pseudometric $sj(p)$ on $SJ(X)$ is complete.  
\end{proof}

For a topological space $X$, by $\IR^X$ we denote the linear lattice of all
real-valued functions on $X$. Three topologies on the set $\IR^X$ will be
considered: the compact-open one, the topology of uniform convergence of 
functions and that of
pointwise convergence. For a function $f\in \IR^X$ let $\|f\|=\sup\{|f(x)|:
x\in X\}\in [0,\infty]$. We denote by $C(X)\subset \IR^X$ the linear lattice 
of all continuous functions on $X$ and by $C_b(X)\subset C(X)$ the Banach 
lattice of
all bounded continuous functions on $X$, equipped with the standard 
sup-norm.  Let $Y\subset X$ be a subset. An operator $T:\IR^Y\to
\IR^X$ is called an extension operator, proveded $T(f)|Y=f$ for every $f\in
\IR^Y$. For a linear operator $T:\IR^X\to \IR^Z$, where $Z$ is a topological
space, let $\|T\|=\inf\{M\in\IR\mid \forall f\in \IR^X\; \|T(f)\|\le
M\,\|f\|\}$ be the norm of the operator $T$. 
The operator $T$ is called {\it regular}, provided it is linear, positive (i.e.
$T(f)\ge 0$ for every $f\ge 0$, $f\in \IR^X$) and $T(1_X)=1_Y$, where
$1_X:X\to \{1\}\subset \IR$ is the constant unit function. It is easily seen that
every regular operator has the unit norm.
The operator $T$ is called
{\it bicontinuous}, provided it is continuous with respect both to the
uniform and the pointwise convergence of functions.
If additionally, the operator $T$ is continuous with respect to the
compact-open topology, it is called {\it tricontinuous}.

Lemmas~\ref{l:1.1}--\ref{l:1.8} and the great magic formula imply

\begin{theorem}\label{t:1.1} $sj:\IR^{X\t X}\to \IR^{SJ(X)\t SJ(X)}$ is a regular
bicontinuous extension operator preserving the classes of bounded (continuous)
functions, (pseudo)metrics, totally bounded pseudometrics, complete bounded
pseudometrics, dominating metrics and the class of bounded admissible
metrics.
\end{theorem}

\section{Infinite iterated squeezed join construction}

Let $X$ be a topological space. Recall that for points $a,b\in X$ the set 
$[a,b]=\{[a,b;t]\mid t\in [0,1]\}$ is called the segment connecting the points
$a$ and $b$. Denote $SJ^0(X)=X$ and
$SJ^n(X)=SJ(SJ^{n-1}(X))$ for $n\ge 1$. Consider the sequence
$$
X\subset SJ^1(X)\subset SJ^2(X)\subset\dots\subset SJ^n(X)\subset \dots
$$
and its union $SJ^\infty(X)=\bigcup_{n=1}^\infty SJ^n(X)$. For points
$a,b\in SJ^{n-1}(X)$ by $[a,b]_n$ we denote the segment in $SJ^n(X)\subset
\SJI$ connecting the points $a$ and $b$. (Remark that for $a\not=b$
$[a,b]_n\not=[a,b]_{n+1}$). We shall say that a set $U\subset \SJI$ is
{\it convex} iff there exists $n\in\IN$ such that $[a,b]_k\subset U$ for
every $k\ge n$ and $a,b\in U\cap SJ^{k-1}(X)$. Now we define the topology
on $\SJI$. The base of this topology consists of convex sets $U\subset \SJI$ 
such that the intersection $U\cap SJ^n(X)$ is open in $SJ^n(X)$ for every
$n\in\IN$.

For a continuous map $f:X\to Y$ beween topological spaces let $SJ^0(f)=f$
and $SJ^n(f)=SJ(SJ^{n-1}(f)):SJ^n(X)\to SJ^n(Y)$ for $n\ge 1$. It is
easily seen that the formula $SJ^\infty(f)(x)=SJ^n(f)(x)$ where $x\in
SJ^n(X)\subset \SJI$, $n\in\IN$, correctly defines the continuous map
$SJ^\infty(f):SJ^\infty(X)\to SJ^\infty(Y)$.

\begin{remark} In spite of the fact that each two points $a,b\in\SJI$
could be connected by a segment, the space $\SJI$ is not naturally
equiconnected, i.e. there is no natural continuous map $\lambda:\SJI\t
\SJI\t [0,1]\to\SJI$ such that $\lambda(a,b,0)=a$, $\lambda(a,b,1)=b$ and
$\lambda(a,a,t)=a$ for every $a,b\in\SJI$ and $t\in [0,1]$. The matter is
that there is a lot of distinct seqments $[a,b]_k$ (for sufficiently great
$k$) connecting the points $a,\;b$ and it is not possible to choose one of
them to construct a continuous equiconnected map $\lambda$ on $\SJI$.
\end{remark}

For every $a\in SJ^\infty(X)$ we can inductively define the support
$\supp(a)\subset X$ of $a$ as follows. For $a\in X=SJ^0(X)$ let
$\supp(a)=\{a\}$. Assuming that for every $a\in SJ^{n-1}(X)$ the support
$\supp(a)\subset X$ has been defined, for $a\in SJ^n(X)\bs SJ^{n-1}(X)$
let $\supp(a)=\supp(x)\cup\supp(y)$, where $a=[x,y;t]$, $x,y\in
SJ^{n-1}(X)$. It is easily seen that for every $a\in SJ^n(X)$ the support
$\supp(a)$ contains no more than $2^n$ points.

For a function $p:X\t X\to\IR$ let $sj^0(p)=p:SJ^0(X)\t SJ^0(X)\to\IR$ and by 
induction let $sj^n(p)=sj(sj^{n-1}(p)):SJ^n(X)\t SJ^n(X)\to\IR$. It is easily 
seen that the formula $\sji(p)(a,b)=sj^n(p)(a,b)$ for 
$a,b\in SJ^n(X)\subset \SJI$, $n\in\IN$, correctly defines a map 
$sj^\infty(p):\SJI\t\SJI\to\IR$ which extends the function $p$. Notice that for
fixed $a,b\in SJ^\infty(X)$ the value of $sj^\infty(p)(a,b)$ depends only on the
values of $p$ on the set $\supp(a)\times \supp(b)$.
Namely, the following statement which can be easily derived from Lemma~\ref{l:1.1} holds:

\begin{lemma}\label{l:2.1} For maps $p,p':X\t  X\to \IR$ and $a,b\in \SJI$ if
$(p-p')|\supp(a)\t \supp(b)\equiv 0$ then $\sji(p)(a,b)=\sji(p')(a,b)$. 
Moreover, if $|p(x,x')|<1$ for every 
$(x,x')\in \supp(a)\t\supp(b)$ then $|\sji(p)(a,b)|<1$.
\end{lemma}

\begin{theorem}\label{t:2.1} $\sji:\IR^{X\t X}\to \IR^{\SJI\t \SJI}$ is a
regular bicontinuous extension operator preserving the class of bounded
(continuous) functions and the class of (pseudo)metrics.
\end{theorem}

\begin{proof} Let's show that for a continuous bounded function $p:X\t X\to
\IR$, the map $\sji(p):\SJI\t\SJI\to\IR$ is continuous. Without loss of
generality, $\|p\|>0$. Fix points $x_1,x_2\in \SJI$ and $\e>0$. There
exists $n\in\IN$ such that $x_1,x_2\in SJ^n(X)$. 
Notice that $\sji(p)(x_1,x_2)=sj^n(p)(x_1,x_2)$. By Lemma~\ref{l:1.2}, the map
$sj^n(p):SJ^n(X)\times SJ^n(X)\to\IR$ is continuous. Hence
there exist neighborhoods $U_1,U_2\subset SJ^n(X)$ of $x_1,x_2$ such that
$sj^n(p)(U_1\t U_2)\subset (sj^n(p)(x_1,x_2)-\e/2,sj^n(p)(x_1,x_2)+\e/2)$. 
For $i=1,2$, let
$U_i^n=U_i$ and inductively $U^k_i=[U_i^{k-1};\frac\e{6\cdot
2^k\|p\|}]\subset SJ^k(X)$ for every $k>n$. Obviously, $U_i^k\subset
U_i^{k+1}$ for $k\ge n$, $i=1,2$. Let $V_i=\bigcup_{k=n}^\infty U^k_i,\;
i=1,2$. It is easily seen that the sets $V_1,V_2\subset \SJI$ are convex
and open. Moreover, Lemma~\ref{l:1.3} implies
$|\sji(p)(a,b)-\sji(p)(x_1,x_2)|=|\sji(p)(a,b)-sj^n(p)(x_1,x_2)|
<\frac\e2+3\|p\|\sum\limits_{k=n}^\infty\dfrac\e{6\cdot
2^k\|p\|}\le\e$ for each $a\in V_1,\; b\in V_2$. Hence the map
$\sji(p):\SJI\t\SJI\to\IR$ is continuous.

The rest statements of Theorem are rather trivial and follow from Lemmas~\ref{l:1.1} and \ref{l:1.4}.
 \end{proof}

Let $d_1,\; d_2$ be two metrics on $X$. We say that the metric $d_1$
{\it dominates} the metric $d_2$, provided for every $x\in X$ and $\e>0$ there
is $\delta>0$ such that $O_{d_1}(x,\delta)\subset O_{d_2}(x,\e)$. The
metrics $d_1,\; d_2$ are called {\it equivalent}, provided they dominate
each other.

\begin{lemma}\label{l:2.2} 
If a metric $d_1$ dominates a bounded metric $d_2$ on $X$ then the metric
$sj^\infty(d_1)$ dominates the metric $sj^\infty(d_2)$ on $SJ^\infty(X)$.
Consequently, if $d_1$ and $d_2$ are equivalent bounded metrics on
$X$ then $\sji(d_1)$ and $\sji(d_2)$ are equivalent metrics on $\SJI$.
\end{lemma}       

\begin{proof} If $\card(X)\le 1$ then Lemma is trivial, so further we
assume $\card(X)>1$.

Suppose that a metric $d_1$ dominates a bounded  metric  $d_2$ on $X$. 
Fix $x\in \SJI$ and $\e>0$. Lemma will be proved if we show that
$O_{\sji(d_1)}(x,\delta)\subset O_{\sji(d_2)}(x,\e)$ for some $\delta>0$.
Find $n\in\IN$ with $x\in SJ^n(X)=SJ^0(SJ^n(X))$. It follows from Lemma
1.5 that the metric $sj^n(d_1)$ dominates the bounded metric $sj^n(d_2)$
(it is sufficiently to consider on $X$ the topology generated by the
metric $d_2$ and apply Lemma~\ref{l:1.6} which implies that then $sj^n(d_2)$ is an
admissible metric on $SJ^n(X)$~).  Without loss of generality,
$x\in SJ^0(X)=X$ (otherwise, we can put $X'=SJ^n(X)$ and consider the metrics
$d_1'=sj^n(d_1)$ and $d_2'=sj^n(d_2)$ on $X'$ in place of the 
metrics $d_1$ and $d_2$ on $X$). Since the metric $d_1$ dominates the metric $d_2$, there
exists $\e_1>0$ such that $O_{d_1}(x,\e_1)\subset O_{d_2}(x,\e/2)$. We
claim that $O_{\sji(d_1)}(x,\delta)\subset O_{\sji(d_2)}(x,\e)$ for
$\delta=\dfrac{\e\cdot\e_1}{2\|d_2\|}$. Indeed, let $\mu\in\SJI$,
$\sji(d_1)(\mu,x)<\delta$. It is easily seen that there exist $n\in \IN$,
reals $t_1,\dots,t_n\in [0,1]$ with $\sum_{i=1}^nt_i=1$ and points
$x_1,\dots,x_n\in X$ such that for every function $p:X\t X\to\IR$ \
$\sji(p)(\mu,x)=\sum_{i=1}^nt_ip(x_i,x)$. Since
$\sji(d_1)(\mu,x)=\sum_{i=1}^nt_id_1(x_i,x)<\delta=\frac{\e\cdot
\e_1}{2\|d_2\|}$, we have $\sum_{d_1(x_i,x)\ge\e_1}t_i<\frac\e{2\|d_2\|}$.
Then
$$
\begin{aligned}
\sji(d_2)(\mu,x)&=\sum_{i=1}^nt_id_2(x_i,x)=\sum_{d_1(x_i,x)<\e_1}t_id_2(x_i,x)+
\sum_{d_1(x_i,x)\ge\e_1}t_id_2(x_i,x)\\
&<\big(\sum_{d_1(x_i,x)<\e_1}t_i\big)\cdot
\e/2+\big(\sum_{d_1(x_i,x)\ge\e_1}t_i\big)\|d_2\|<\e/2+\frac\e{2\|d_2\|}\|d_2\|=\e
\end{aligned}
$$
(we use here that $d_2(x_i,x)<\e/2$, provided $d_1(x_i,x)<\e_1$).  
\end{proof}


\section{Regular operators extending metrics}

Recall that a Hausdorff space $X$ is defined to be {\it stratifiable} if
to every open set $U\subset X$ it is assigned a sequence
$\{U_n\}_{n=1}^\infty$ of open subsets of $X$ such that (a) $\bar
U_n\subset U$, (b) $\bigcup_{n=1}^\infty U_n=U$ and (c) $U_n\subset V_n$
whenever $U\subset V$, $n\in\IN$. The class of stratifiable spaces
contains the class of all metrizable spaces and possesses many remarkable
properties (see [6]). In particular, each stratifiable space is perfectly
paracompact, each subset of a stratifiable space is stratifiable too.

We shall say that the space is {\it non-degenerate}, provided it contains more
than one point. 
Recall that an operator is tricontinuous, provided it is continuous with
respect to the uniform, pointwise and compact-open convergences of
functions.
The following Theorem is the main result of this paper.

\begin{theorem}\label{t:3.1} Let $Y$ be a stratifiable (metrizable) space and $X$ 
be a closed
non-degenerate set in $Y$. There exists a regular tricontinuous
extension operator $T:\IR^{X\times X}\to \IR^{Y\times Y}$
preserving the classes of bounded functions, continuous functions,
pseudometrics, metrics (dominating metrics, and the class of admissible 
metrics).
Moreover, the operator $T$ can be constructed so that one of the following 
conditions is satisfied:
\begin{enumerate}
\item if \/ $Y$ is complete-metrizable and $X$ is compact then $T$ preserves the 
class of complete admissible metrics;
\item if $Y$ is separable and $\dim Y\bs X<\infty$ then $T$ preseves the class
of totally bounded pseudometrics.
\end{enumerate}
\end{theorem}

For a space $X$ by $\exp_\omega X$ we denote the hyperspace of all finite
subsets of $X$. A map $u:Y\to \exp_\omega X$ is called upper-semicontinuous
provided for every open set $U\subset X$ the set $\{y\in Y\mid u(y)\subset
U\}$ is open in $Y$. Let us prove at first the following important

\begin{lemma}\label{l:3.1} For every stratifiable space $Y$ and a closed subset
$X\subset Y$ there exist a continuous map $f:Y\to \SJI$ extending the
identity embedding $X\to\SJI$, and an upper semi-continuous map $u:Y\to \exp_\omega X$
such that $u(x)=\{x\}$ for $x\in X$ and for every $y\in Y$ \
$\supp(f(y))\subset u(y)$.
Moreover, if $\dim(Y\bs X)\le n<\infty$, then $f(Y)\subset SJ^n(X)\subset \SJI$.
\end{lemma}

\begin{proof} Let $Y$ be a stratifiable space and $X$
be a closed subspace. In the proof of Theorem 4.3 [6], Carlos
J.R.~Borges has constructed a locally finite cover $\U$ of $Y\bs X$ and a
map $a:\U\to X$ such that the map $u:Y\to \exp_\omega(X)$ defined by
$u(y)=\{y\}$ for $y\in X$ and $u(y)=\{a(U)\mid y\in \operatorname{cl}(U),\; U\in\U\}$ for $y\in
Y\bs X$ is upper semi-continuous.

Let  $N(\U)$ be the nerve of the cover $\U$, $\{\lambda_U:Y\bs X\to
[0,1]\}_{U\in\U}$ be a partition of unity subordinated to the cover $\U$
and $g:Y\bs X\to N(\U)$ be the map of $Y\bs X$ into $N(\U)$ acting by
$g(y)=\sum_{U\in\U}\lambda_U(y)\cdot U$, $y\in Y\bs X$.
By $N(\U)^{(n)}$ the $n$-skeleton of $N(\U)$ is denoted.

We shall construct inductively the map $h:N(\U)\to \SJI$ as follows: for
$U\in N(\U)^{(0)}=\U$ let $h(U)=a(U)$. Assume that for $n\ge 1$
the map $h:N(\U)^{(n-1)}\to SJ^{n-1}(X)\subset\SJI$ has been defined on the
$(n-1)$-skeleton $N(\U)^{(n-1)}$ of $N(\U)$. We shall show how to extend $h$
onto $N(\U)^{(n)}$. Let $\sigma\in N(\U)^{(n)}$ be an $n$-simplex and $b\in\sigma$
be its barycenter. Let $h(b)=h(v)$, where $v\in N(\U)^{(0)}$ is any of vertices
of $\sigma$. Each point $x\in \sigma\bs\{b\}$ can be expressed in the
unique  way as $x=(1-t)y+tb$, where $t\in [0,1]$ and $y\in \partial
\sigma$ ($\partial\sigma$ is the boundary of $\sigma$). Let
$h(x)=[h(y),h(b);t]_n\in SJ^n(X)$.

Proceeding inductively, we shall construct the continuous map $h:N(\U)\to
\SJI$ such that $h(N(\U)^{(n)})\subset SJ^n(X)$ for every $n\ge 0$. Define the
map $f:Y\to \SJI$ by the formula
$$
f(y)=\begin{cases}
y, & \text{if }y\in X\\
h\circ g(y), & \text{if }y\in Y\bs X.
\end{cases}
$$

Obviously that the map $f$ is continuous on the set $Y\bs X$. Let us
verify continuity of $f$ at points of $X$. Fix $x\in X$ and a convex open
neighborhood $U\subset\SJI$ of $f(x)=x$. Since the set $U$ is convex there
exists $n\in\IN$ such that $[a,b]_k\subset U$ for any $k\ge n$ and $a,b\in
U\cap SJ^{k-1}(X)$. It follows from the definition of topology on the
squeezed join that there exists a neighborhood $W_0\subset X$ of $x$ such
that $W_n\subset U$, where $W_k=\bigcup\{[a,b]_k\mid a,b\in W_{k-1}\}$, $1\le
k\le n$. 
Since the map $u:Y\to \exp_\omega X$ is upper-semicontinuous and $u(x)=\{x\}$,
there is a neighborhood $W\subset Y$ of $x$ such that $u(y)\subset W_0$
for every $y\in W$. Now it follows from the construction of the map $f$
and convexity of $U$ that $f(y)\in U$ for all $y\in W$, i.e. $f:Y\to \SJI$
is a continuous map. Moreover, it is clear that $\supp(f(y))\subset u(y)$
for every $y\in Y$.

If $\dim(Y\bs X)\le n$ then the
cover $\U$ can be chosen of order $\le n+1$. In this case $N(\U)^{(n)}=N(\U)$
and $f(Y)\subset h(N(\U)^{(n)})\cup X\subset SJ^n(X)$.  
\end{proof}

Now we can complete the proof of Theorem~\ref{t:3.1}. 
The operator $T$ will be constructed as the sum of the series
$\sum_{n=1}^\infty \frac1 {2^n}T_n$, where the collection of extension operators
$\{T_n:\IR^{X\t X}\to \IR^{Y\t Y}\}_{n=1}^\infty$  ``separates"
points of the space $Y$.

It is well known that every stratifiable space admits a bijective
continuous map onto a metrizable space (to prove this just apply [6, Lemma
8.2] and property (A) on p.2 [6]). Therefore, there is a continuous metric
$d\le 1$ on $Y$. If $Y$ is metrizable then $d$ will be assumed 
to be admissible. Depending on the case considered, we will assume that either
$d$ is complete (if $Y$ is complete-metrizable) or $d$ is totally bounded
(if $Y$ is separable).

By Lemma~\ref{l:3.1}, there exist a map $h:Y\to \SJI$ extending the identity
embedding $X\hookrightarrow \SJI$ and an upper semi-continuous map
$u:Y\to\exp_\omega X$ with $u(x)=\{x\}$ for $x\in X$ such that for every
$y\in Y$ \ $\supp(h(y))\subset u(y)$ (if $\dim Y\bs X<\infty$ we can
assume that $h(Y)\subset SJ^k(X)$ for some $k\in\IN$). Consider the  map
$\chi:Y\to [0,2]$ defined by
$\chi(y)=\inf\{d(y,x)+sj^\infty(d_X)(h(y),x)|x\in X\}$, $y\in Y$.
Let us show that the map $\chi$ is continuous. Fix $y_0\in Y$ and $\e>0$.
It follows from Theorem~\ref{t:2.1} that $\sji(d_X)$ is a continuous pseudometric 
on $\SJI$. Since $h:Y\to \SJI$ is continuous, we can find a neighborhood 
$W\subset Y$ of $y_0$ such that for every $y\in W$ \
$\sji(d_X)(h(y),h(y_0))<\frac\e3$. Moreover, since the metric $d$ is continuous,
$W$ can be chosen so small that $d(y,y_0)<\frac\e3$ for every $y\in W$. We claim
that $|\chi(y)-\chi(y_0)|<\e$ for every $y\in W$. Indeed, by the definition of 
$\chi(y_0)$, there is $x\in X$ with $\chi(y_0)+\frac\e3>d(y_0,x)+
\sji(d_X)(h(y_0),x)$. Then $\chi(y)\le d(y,x)+\sji(d_X)(h(y),x)\le
d(y_0,x)+d(y_0,y)+\sji(d_x)(h(y_0),x)+\sji(d_X)(h(y),h(y_0))<\chi(y_0)+\e$.
To see that $\chi(y)> \chi(y_0)-\e$, notice that for every $x\in X$
$d(y,x)+\sji(d_X)(h(y),x)\ge d(y_0,x)-d(y,y_0)+\sji(d_X)(h(y_0),x)-
\sji(d_X)(h(y),h(y_0))\ge \chi(y_0)-\frac23\e$. Hence, $\chi(y)=\inf\{d(y,x)+
\sji(d_X)(h(y),x)\mid x\in X\}\ge \chi(y_0)-\frac23\e>\chi(y_0)-\e$.

For every $n\in\IN$, define the map $\chi_n:Y\to [0,1]$
by $\chi_n(y)=\min\{1,n\,\chi(y)\}$, $y\in Y$, and notice that
$\chi_n^{-1}(0)=X$ and $\chi_n^{-1}(1)=\chi^{-1}([\frac1n,2])$.

Fix any two distinct points $a,b\in X$. For every $n\in\IN$ we shall define
an extension operator $T_n:\IR^{X\t X}\to\IR^{Y\t Y}$ as follows. Fix $n\in\IN$.
Let $\U_n$ be a locally finite (finite, if the metric $d$ is totally
bounded) open 
cover  of the space $Y$ such that $\diam(U)< 2^{-n}$ for
every  $U\in \U_n$, and let $\{\lambda^n_U:Y\to [0,1]\}_{U\in\U_n}$ be a
partition of unity, subordinated to the cover $\U_n$. Let $N(\U_n)$ be the
nerve of the cover $\U_n$ and $g_n:Y\to N(\U_n)$ be the map, defined by
$g_n(y)=\sum_{U\in\U_n}\lambda^n_U(y)\cdot U$ for $y\in Y$.

Futher we shall consider the set $\U_n$ as a discrete topological space.
By analogy with the proof of Lemma~\ref{l:3.1}, we shall construct the map
$q_n:N(\U_n)\to SJ^\infty(\U_n)$ as follows. For every $U\in N(\U_n)^{(0)}$ let
$q_n(U)=U\in SJ^\infty(\U_n)$.
 Assume that for some $m\in\IN$ the map $q_n$ has been
defined on the $(m-1)$-skeleton $N(\U_n)^{(m-1)}\subset N(\U_n)$. We shall show
how to extend the map $q_n$ onto $N(\U_n)^{(m)}$. Let $\s\in N(\U_n)^{(m)}$ be an
$m$-simplex and $b\in \s$ be its barycenter. Let $q_n(b)=q_n(v)$, where
$v\in N(\U_n)^{(0)}$ is any of vertices of $\s$. Each point $x\in\s\bs \{b\}$
can be expressed in the unique way as $x=(1-t)y+tb$, where $t\in [0,1]$
and $y\in\partial \s$. Let $q_n(x)=[q_n(y),q_n(b);t]_m\in SJ^m(\U_n)$.

By $X\sqcup \U_n$  denote the disjoint union of the
spaces $X$ and $\U_n$, and consider the maps $SJ^\infty(i_X):SJ^\infty(X)\to
SJ^\infty(X\sqcup \U_n)$ and $SJ^\infty(i_n):SJ^\infty(\U_n)\to
SJ^\infty(X\sqcup \U_n)$, where $i_X:X\to X\sqcup \U_n$ and $i_n:\U_n\to
X\sqcup \U_n$ are the embeddings.
Define the map $f_n:Y\to
SJ(SJ^\infty(X\sqcup\U_n))$ by
$f_n(y)=[SJ^\infty(i_X)(h(y)),SJ^\infty(i_n)(q_n(y));\chi_n(y)],\; y\in Y$.

Let us consider the linear
operator $E_n:\IR^{X\t X}\to \IR^{(X\sqcup\U_n)\t (X\sqcup \U_n)}$ defined for
every $p\in \IR^{X\t X}$ by
$$
E_n(p)(x,y)=\begin{cases}
p(x,y),&\text{if $x,y\in X$;}\\
\frac12p(x,a)+\frac12p(x,b),&\text{if $x\in X,\; y\in\U_n$;}\\
p(a,b),&\text{if $x,y\in \U_n$ and $x\ne y$}\\
0, & \text{if $x=y$};
\end{cases}
$$
(recall that $a,b$ are two fixed point in $X$).
One can easily verify that the operator $E_n$ is regular, tricontinuous
and preserves the classes of bounded functions, continuous functions,
(pseudo)metrics, dominating metrics and the  class of admissible metrics.

Now, let us consider the composition
$$
G_n=sj\circ\sji\circ E_n:\IR^{X\t X}\to \IR^{SJ(SJ^\infty(X\sqcup \U_n))\t
SJ(SJ^\infty(X\sqcup \U_n))}
$$
and the operator $T_n:\IR^{X\t X}\to \IR^{Y\t Y}$ defined for 
$p\in \IR^{X\t X}$ by
the formula $$T_n(p)(y,y')=G_n(p)(f_n(y),f_n(y')),\quad y,y'\in Y.$$ Remark
that $\|T_n\|$'s are  regular bicontinuous extension operators
with the unit norm. Moreover, by Theorems~\ref{t:1.1}, \ref{t:2.1}, the operators $T_n$'s
preserve the class of bounded (continuous) functions and the class of
pseudometrics.

Finally, let $T:\IR^{X\t X}\to \IR^{Y\t Y}$ be the operator defined by the
formula $T=\sum_{n=1}^\infty \frac1{2^n}T_n$.

We claim that the operator $T$ satisfies all the conditions of Theorem.
Since all $T_n$'s are extension operators, so is the operator $T$.

Let us show that the definition of $T$ is correct, i.e. for every
function $p:X\t X\to \IR$ and every $y,y'\in Y$ the series
$\sum_{n=1}^\infty \frac1{2^n}T_n(p)(y,y')$ is convergent. This is
trivial, when $y,y'\in X$. If $y\in X$ and $y'\notin X$ then for every
$n\in\IN$ with $n\,\chi(y')\ge 1$, by the construction of $T_n$, we have
$T_n(p)(y,y')=\frac12p(y,a)+\frac12 p(y,b)$. If $y,y'\notin X$ then,
for every $n\in\IN$ with $n\,\chi(y)\ge 1$ and $n\,\chi(y')\ge 1$,
$|T_n(p)(y,y')|\le |p(a,b)|$. These remarks imply that the series
$\sum_{n=1}^\infty \frac1{2^n}T_n(y,y')$  converges for every $y,y'\in Y$,
i.e. the definition of $T$ is correct.

Since $\|T_n\|=1$ for every $n\in\IN$, it follows
$\|T\|=\|\sum_{n=1}^\infty\frac1{2^n}T_n\|=1$. Since each operator $T_n$
preserves the class of bounded (continuous) functions and the class of
pseudometrics, so does the operator $T$.

In sake of the reader's conveniency the further proof is divided onto the
following lemmas.

\begin{lemma}\label{l:3.2} Let $A\subset X$ be any subset of $X$ containing the
fixed points $a,b$, and $y,y'\in Y$ be  such that $\supp(h(y))\cup
\supp(h(y'))\subset A$. Then for every functions $p,p':X\t X\to \IR$ if
$p|A\t A\equiv p'|A\t A$ then $T(p)(y,y')=T(p')(y,y')$. Moreover, if
$|p(x,x')|\le 1$ for every $x,x'\in A$ then $|T(p)(y,y')|\le 1$.
\end{lemma}

\begin{proof} Let $p,p':X\t X\to \IR$ be two functions with $p|A\t A=p'|A\t
A$. We shall show that $T_n(p)(y,y')=T_n(p')(y,y')$ for every $n\in \IN$.
To simplify the denotations, we shall identify the points $h(y),
h(y')\in\SJI$, $q_n(y),q_n(y')\in SJ^\infty (\U_n)$ with their images
$$SJ^\infty(i_X)(h(y)), SJ^\infty(i_X)(h(y')), 
SJ^\infty(i_n)(q_n(y)), SJ^\infty(i_n)(q_n(y'))$$ in $SJ^\infty(X\sqcup
\U_n)$ (recall that $i_X:X\to X\sqcup \U_n$ and $i_n:\U_n\to X\sqcup \U_n$
are the natural embeddings). It follows from the definition of functions
$E_n(p)$ and $E_n(p')$ that $(p-p')|A\t A\equiv 0$ and $\{a,b\}\subset A$
imply that $(E_n(p)-E_n(p')|A\t A\cup A\t \U_n\cup \U_n\t A\cup \U_n\t
\U_n\equiv 0$. Now remark that $\supp(h(y))\cup\supp(h(y'))\subset A$ and
$\supp(q_n(y))\cup \supp(q_n(y'))\subset \U_n$. Therefore, by Lemma~\ref{l:2.1},
$$
\begin{gathered}
\sji(E_n(p))(h(y),h(y'))=\sji(E_n(p'))(h(y),h(y'))\\
\sji(E_n(p))(h(y),q_n(y'))=\sji(E_n(p'))(h(y),q_n(y'))\\
\sji(E_n(p))(q_n(y),h(y'))=\sji(E_n(p'))(q_n(y),h(y'))\\
\sji(E_n(p))(q_n(y),q_n(y'))=\sji(E_n(p'))(q_n(y),q_n(y'))
\end{gathered}
$$
Then
$$
\begin{aligned}
T_n(p)(y,y')&=G_n(p)(f_n(y),f_n(y'))\\
&=sj(\sji(E_n(p)))([h(y),q_n(y);\chi_n(y)],[h(y'),q_n(y');\chi_n(y')])\\
&=\min\{1-\chi_n(y),1-\chi_n(y')\}\sji(E_n(p))(h(y),h(y'))\\
&+\max\{0,\chi_n(y')-\chi_n(y)\}\sji(E_n(p))(h(y),q_n(y'))\\
&+\max\{0,\chi_n(y)-\chi_n(y')\}\sji(E_n(p))(q_n(y),h(y'))\\
&+\min\{\chi_n(y),\chi_n(y')\}\sji(E_n(p))(q_n(y),q_n(y'))\\
&=\min\{1-\chi_n(y),1-\chi_n(y')\}\sji(E_n(p'))(h(y),h(y'))\\
&+\max\{0,\chi_n(y')-\chi_n(y)\}\sji(E_n(p'))(h(y),q_n(y'))\\
&+\max\{0,\chi_n(y)-\chi_n(y')\}\sji(E_n(p'))(q_n(y),h(y'))\\
&+\min\{\chi_n(y),\chi_n(y')\}\sji(E_n(p'))(q_n(y),q_n(y'))\\
&=sj(\sji(E_n(p')))([h(y),q_n(y);\chi_n(y)],[h(y'),q_n(y');\chi_n(y')])\\
&=G_n(p')(f_n(y),f_n(y'))=T_n(p')(y,y').
\end{aligned}
$$

Therefore, we have proven that for functions $p,p':X\t X\to\IR$ the
equality $p|A\t A\equiv p'|A\t A$ implies that
$T_n(p)(y,y')=T_n(p')(y,y')$ for every $n\in\IN$. Since
$T=\sum_{n=1}^\infty \frac1{2^n}T_n$ this immediately yields
$T(p)(y,y')=T(p')(y,y')$.

Now assume that for a function $p:X\t X\to \IR$ \ $|p(x,x')|<1$ for every
$(x,x')\in A$. We shall show that $|T(p)(y,y')|<1$. Indeed, define the map 
$p':X\t X\to \IR$ by the formula
$$p'(x,x')=\begin{cases} p(x,x'),& \text{if $(x,x')\in A\times A$};\\
0,&\text{otherwise.}\end{cases}
$$
It follows from the above discussion that $T(p)(y,y')=T(p')(y,y')$. 
Since $\|T\|=1$ and $\|p'\|\le 1$
we have $|T(p')(y,y')|=|T(p)(y,y')|\le
\|T\|\cdot\|p'\|\le1$.  
\end{proof}

\begin{lemma}\label{l:3.3} The operator $T$ is tricontinuous.
\end{lemma}

\begin{proof} Notice that since $T$ is a regular operator, it is continuous
with respect to the uniform convergence of functions.

Let us show that the operator $T$ is continuous with respect to the
pointwise convergence of functions. For this, fix points $y,y'\in Y$ and
notice that  the set $A=\{a,b\}\cup\supp(h(y))\cup\supp(h(y'))$ is finite.
By Lemma~\ref{l:3.2},  for a function $p:X\t X\to \IR$ the inequality $|p(x,x')|\le 1$ for
every $(x,x')\in A\t A$ implies $|T(p)(y,y')|\le 1$. This means that the
operator $T$ is continuous with respect to the pointwise convergence of
functions.

To show that  $T$ is continuous with respect to
the compact-open topology fix a compactum $C\subset Y\t Y$ and notice that
the set $K'=\bigcup\{u(y)\mid y\in \pr_1(C)\cup \pr_2(C)\}\subset X$ is
compact because of upper-semicontinuity of the map $u:Y\to\exp_\omega X$
(see [10, Theorem VI.7.10]) (by $\pr_i:Y\t Y\to Y$ we denote the projection onto the
corresponding factor). Consider the compact set $K=K'\cup\{a,b\}$. Since
for every $y\in Y$ \ $\supp(h(y))\subset u(y)$, $\supp(h(y))\cup
\supp(h(y'))\subset u(y)\cup u(y')\subset K$ for every $(y,y')\in C$. Then
Lemma~\ref{l:3.2} yields that for a function $p:X\t X\to \IR$ if $|p(x,x')|\le 1$
for every $(x,x')\in K\t K$ then $|T(p)(y,y')|\le 1$ for every $(y,y')\in
C$. But this means that the operator $T$ is continuous with respect to the 
compact-open topology.  
\end{proof}

\begin{lemma}\label{l:3.4} The operator $T$ preserves continuous functions.
\end{lemma}

\begin{proof} Let $p:X\t X\to \IR$ be a continuous function. Fix any point
$(y_0,y_0')\in Y\t Y$. Let $M=\max\{|p(x,x')|: x,x'\in \{a,b\}\cup u(y_0)\cup 
u(y_0')\}$. Since the map $p$ is continuous, there is a neighborhood $U\subset X$
of the compactum $\{a,b\}\cup u(y_0)\cup u(y_0')$ such that $|p(x,x')|<M+1$
for every $x,x'\in U$. Since the map $u:Y\to\exp_\omega X$ is upper-semicontinuous,
there are neighbourhoods $V,V'\subset Y$ of $y_0,y_0'$ respectively such that
for every $y\in V$ and $y'\in V'$ we have $u(y)\cup u(y')\subset U$. 

Now consider the bounded continuous function $\tilde p:X\t X\to \IR$
defined by the formula
$$
\tilde p(x,x')=\begin{cases} p(x,x'), & \text{ if } -M-1\le p(x,x')\le M+1\\
M+1, &\text{ if } p(x,x')\ge M+1\\
-M-1, &\text{ if } p(x,x')\le -M-1.\end{cases}
$$
Obviously that $\tilde p|U= p|U$. Moreover, since the operator
$T$ preserves bounded continuous functions, the map $T(\tilde
p):Y\t Y\to\IR$ is continuous. Now remark that for every $(y,y')\in V\t V'$ \ 
$\supp(h(y))\cup\supp(h(y'))\subset \{a,b\}\cup u(y)\cup u(y')\subset U$. Since $\tilde
p|U\t U=p|U\t U$, by Lemma~\ref{l:3.2}, $T(p)(y,y')=T(\tilde p)(y,y')$. Therefore,
$T(p)|V\t V'=T(\tilde p)|V\t V'$ and the function $T(p)$ is continuous. 
\end{proof}

\begin{lemma}\label{l:3.5} The operator $T$ preserves the cone of metrics.
\end{lemma}

\begin{proof} Let $p$ be a metric on $X$. Since the operator $T$ preserves
the cone of pseudometrics, the only we have to prove is that
$T(p)(y,y')\not= 0$ for distinct $y,y'\in Y$. So, fix $y,y'\in Y$ with
$y\not=y'$.

If $y,y'\in X$ then $T(p)(y,y')=p(y,y')\ne 0$ because $p$ is a metric on
$X$. Now assume that $y\in X$ and $y'\notin X$. Since
$\chi^{-1}(\{0\})=X$, $\chi(y')\ne0$. Hence, $\chi_n(y)=0$ and
$\chi_n(y')=1$ for some $n\in\IN$. Then
$$
\begin{aligned}
T(p)(y,y')&\ge 2^{-n}T_n(p)(y,y')=2^{-n}G_n(p)(f_n(y),f_n(y'))\\
&=2^{-n}sj(\sji(E_n(p)))([h(y),q_n(y);\chi_n(y)],[h(y'),q_n(y');\chi_n(y')])\\
&=2^{-n}sj(\sji(E_n(p)))([h(y),q_n(y);0],[h(y'),q_n(y');1])\\
&=2^{-n}\sji(E_n(p))(h(y),q_n(y')).
\end{aligned}
$$
Now notice that $\supp(h(y))\subset X$, $\supp(q_n(y'))\subset \U_n$ and for
every $(x,x')\in X\t \U_n$ \ $E_n(p)(x,x')=\frac12p(x,a)+\frac12p(x,b)\ge
\frac12p(a,b)$. By Theorem~\ref{t:2.1}, the operator $\sji:\IR^{(X\sqcup\U_n)^2}\to
\IR^{(SJ^\infty(X\sqcup \U_n))^2}$ is positive and sends the unit function on
$(X\sqcup \U_n)^2$ to the unit function on $(SJ^\infty(X\sqcup \U_n))^2$.
Consequently, by Lemma~\ref{l:2.1}, $\sji(E_n(p))(h(y),q_n(y'))\ge
\sji(\frac12p(a,b))(h(y),q_n(y'))=\frac12p(a,b)>0$. Hence, $T(p)(y,y')\ge
2^{-n-1}p(a,b)>0$.

Now assume that $y,y'\in Y\bs X$. Then there is an $n\in\IN$ such that
$\chi_n(y)=\chi_n(y')=1$ and $d(y,y')>2^{-n+1}$. Since $\diam(U)<2^{-n}$
for every $U\in\U_n$ there is no $U\in \U_n$ with $\{y,y'\}\subset U$.
Consequently, $\supp(q_n(y))\cap \supp(q_n(y'))=\emptyset$. By the
definition of the metric $E_n(p)$, $E_n(p)(x,x')=p(a,b)$ for any
$x,x'\in\U_n\subset X\sqcup \U_n$, $x\ne x'$. Hence,
$E_n(p)|\supp(q_n(y))\t\supp(q_n(y'))\equiv p(a,b)$. By Lemma~\ref{l:2.1},
$\sji(E_n(p))(q_n(y),q_n(y'))=p(a,b)$. Then
$$
\begin{aligned}
T(p)(y,y')&\ge 2^{-n}T_n(p)(y,y')=2^{-n}G_n(p)(f_n(y),f_n(y'))\\
&=2^{-n}sj(\sji(E_n(p)))([h(y),q_n(y);\chi_n(y)],[h(y'),q_n(y');\chi_n(y')])\\
&=2^{-n}sj(\sji(E_n(p)))([h(y),q_n(y);1],[h(y'),q_n(y');1])\\
&=2^{-n}\sji(E_n(p))(q_n(y),q_n(y'))=2^{-n}p(a,b)>0.
\end{aligned}
$$
Therefore, $T(p)$ is
a metric on $Y$.  
\end{proof}

\begin{lemma}\label{l:3.6} If the metric $d$ on
$Y$ is admissible then the operator $T$ preserves the class of 
dominating metrics.
\end{lemma}

\begin{proof} 
Let $p$ be a dominating metric for $X$. 
Let us prove that
 the metric $T(p)$ dominates the topology of $Y$. Fix
$y\in Y$ and its neighborhood $V\subset Y$. We have to find $\e>0$ such
that $O_{T(p)}(y,\e)\subset V$.

Consider at first the case $y\in Y\bs X$. Since $\chi^{-1}(\{0\})=X$, there
is $n\in\IN$ such that $\chi(y)\ge \frac1n$, i.e. $\chi_m(y)=1$ for every
$m\ge n$. Since the metric $d$ on $Y$ is admissible, there is $m\ge n$
such that $O_d(y,2^{-m+1})\subset V$. We claim that for
$\e=2^{-m-1}p(a,b)$ we have $O_{T(p)}(y,\e)\subset V$. Indeed, let $y'\in
Y\bs V$ be any point. Then $d(y,y')\ge 2^{-m+1}$. Since for every
$U\in\U_m$ \ $\diam(U)<2^m$, $\supp(q_m(y))\cap\supp(q_m(y'))=\emptyset$.
As we have already proved in Lemma~\ref{l:3.5}, this implies
$\sji(E_m(p))(q_m(y),q_m(y'))\ge p(a,b)$. Moreover, since for every $x\in
X$ and $z\in\U_m$ \ $E_m(p)(x,z)\ge \frac12p(a,b)$,
$\sji(E_m(p))(q_m(y),c)\ge \frac12p(a,b)$ for every $c\in \SJI\subset
SJ^\infty(X\sqcup \U_m)$. Then
$$
\begin{aligned}
T(p)(y,y')&\ge 2^{-m}T_m(p)(y,y')= 2^{-m}G_m(p)(f_m(y),f_m(y'))\\
&=2^{-m}sj(\sji(E_m(p)))([h(y),q_m(y);\chi_m(y)],[h(y'),q_m(y');\chi_m(y')])\\
&=2^{-m}sj(\sji(E_m(p)))([h(y),q_m(y),1],[h(y'),q_m(y');\chi_m(y')])\\
&=2^{-m}(1-\chi_m(y'))\sji(E_m(p))(q_m(y),h(y'))
+2^{-m}\chi_m(y')\sji(E_m(p))(q_m(y),q_m(y'))\\
&\ge 2^{-m}\cdot \tfrac12\cdot p(a,b),
\end{aligned}
$$
i.e. $y'\notin O_{T(p)}(y',\e)$.

Now assume that $y\in X$. Pick up $0<\delta<1$ such that
$O_d(y,\delta)\subset V$. By Lemma~\ref{l:2.2}, the metric $\sji(p)$ on $\SJI$
dominates the metric $\sji(d_X)$. Hence there exists $\delta_1>0$ such
that $O_{\sji(p)}(y,\delta_1)\subset O_{\sji(d_X)}(y,\delta/2)$. We claim
that for $\e=\min\{\frac{\delta_1}4, \frac\delta{16p(a,b)}\}$ \
$O_{T(p)}(y,\e)\subset V$. Indeed, if $T(p)(y',y)<\e$ then $\frac12 
T_1(p)(y,y')<\e$ and consequently,
$$
\begin{aligned}
T_1(p)(y,y')&=G_1(p)(f_1(y),f_1(y'))\\
&=sj(\sji(E_1(p)))([h(y),q_1(y);\chi_1(y)],[h(y'),q_1(y');\chi_1(y')])\\
&=sj(\sji(E_1(p)))([y,q_1(y);0],[h(y'),q_1(y');\chi_1(y')])\\
&=(1-\chi_1(y'))\sji(E_1(p))(y,h(y'))+\chi_1(y')\sji(E_1(p))(y,q_1(y'))<2\e.
\end{aligned}
$$
As we have already noticed in the proof of Lemma~\ref{l:3.5},
$\sji(E_1(p))(y,q_1(y'))>\frac12p(a,b)$. Then inequality
$(1-\chi_1(y'))\sji(E_1(p))(y,h(y'))+\chi_1(y')\sji(E_1(p))(y,q_1(y'))<2\e$
implies (i) $\chi_1(y)<\frac{4\e}{p(a,b)}\le\frac\delta4<\frac12$ and
$\chi(y')=2\chi_1(y')<\frac\delta2$, and (ii)
$\sji(E_1(p))(y,h(y'))<4\e\le\delta_1$. Notice that since $E_1(p)|X\t
X=p|X\t X$ and $y\in X$, $\supp(h(y'))\subset X$, we have
$\sji(p)(y,h(y'))=\sji(E_1(p))(y,h(y'))<\delta_1$, and hence, 
$\sji(d_X)(h(y'),y)<\delta/2$. Let us show that
$d(y',y)<\delta$. Recall that
$\chi(y')=\inf\{d(y',x)+\sji(d_X)(h(y'),x)\mid x\in X\}$. Since
$\chi(y')<\frac\delta2$, there is $x\in X$ such that
$d(y',x)+\sji(d_X)(h(y'),x)<\frac\delta2$. Then $d(y,y')\le
d(y',x)+d(x,y)=d(y',x)+\sji(d_X)(x,y)\le
d(y',x)+\sji(d_X)(x,h(y'))+\sji(d_X)(h(y'),y)<\frac\delta2+\frac\delta2=\delta$.
Since $O_d(y,\delta)\subset V$ we have $y'\in O_d(y,\delta)\subset V$.
This is a contradiction, hence the metric $T(p)$ on $Y$ is dominating.  
\end{proof}

\begin{lemma}\label{l:3.7} If the metric $d$ on $Y$ is complete and $X$ is
compact then for every
continuous metric $p$ on $X$ \ $T(p)$ is a complete bounded admissible
metric on $Y$.
\end{lemma}

\begin{proof} Let $p$ be any continuous metric on $X$. Since $X$ is
compact, the metric $p$ is bounded and admissible. Then by Lemmas~\ref{l:3.4}--\ref{l:3.6},
$T(p)$ is a bounded admissible metric on $Y$. To show that it is complete
fix any sequence $\{y_n\}_{n=1}^\infty\subset Y$ which is Cauchy with respect
to the metric $T(p)$.
There are two cases:
 $\varliminf\limits_{n\to\infty}\chi(y_n)=
\lim\limits_{n\to\infty}\inf\{\chi(y_k)\mid k\ge n\}=0$
and $\varliminf\limits_{n\to\infty}\chi(y_n)>0$. Assume at first that
$\varliminf\limits_{n\to\infty}\chi(y_n)=0$. Going to a subsequence, we
can assume that $\lim_{n\to\infty}\chi(y_n)=0$. Now recall that
$\chi(y_n)=\inf\{d(y_n,x)+\sji(d_X)(h(y_n),x)\mid x\in X\}$. Then there is
a sequence $\{x_n\}_{n=1}^\infty\subset X$ such that
$\lim_{n\to\infty}d(x_n,y_n)\le\lim_{n\to\infty}\chi(y_n)=0$. Since the
set $X$ is compact, there is $x_0\in X$ a cluster point for the sequence
$\{x_n\}_{n=1}^\infty\subset X$. Again going to a subsequence, we can
assume that the sequence $\{x_n\}_{n=1}^\infty$ converges to $x_0$. Since
$\lim_{n\to\infty}d(y_n,x_n)=0$, the sequence $\{y_n\}_{n=1}^\infty$ also
converges to $x_0$. Since the metric $T(p)$ on $Y$ is admissible, $x_0$ is
the limit pomit of the Cauchy sequence $\{y_n\}_{n=1}^\infty$.

Now suppose that $\varliminf\limits_{n\to\infty}\chi(y_n)\ge\delta>0$.
Without loss of generality, $\delta\le 1$ and $\chi(y_m)\ge \delta$ for
every $m\in\IN$. Let us show that the sequence $\{y_n\}_{n=1}^\infty$ is
Cauchy with respect to the metric $d$. Assuming the converse, find
$n\in\IN$ and a subsequence $\{z_k\}_{k=1}^\infty\subset
\{y_n\}_{n=1}^\infty$ such that $d(z_k, z_m)>2^{-n+1}$ for every distinct
$k,m\in\IN$. Since $\diam(U)<2^{-n}$ for every $U\in\U_n$,
$\supp(q_n(z_k))\cap \supp(q_n(z_m))=\emptyset$ for any distinct $k,m$.
Moreover, it follows from Lemma~\ref{l:3.2} and the definition of the metric
$E_n(p)$ that $\sji(E_n(p))(q_n(z_k),q_n(z_m))=p(a,b)$ for any distinct
$k,m\in\IN$. Since $\chi_n(z_k)=\max\{1,n\cdot\chi(z_k)\}\ge \delta$ for
every $k\in\IN$, this yields that
$$
\begin{aligned}
T_n(p)(z_k,z_m)&=G_n(p)(f_n(z_k),f_n(z_m))\\
&=sj(\sji(E_n(p))([h(z_k),q_n(z_k);\chi_n(z_k)],[h(z_m),q_n(z_m);\chi_n(z_m)])\\
&\ge\min\{\chi_n(z_k),\chi_n(z_m)\}\sji(E_n(p))(q_n(z_k),q_n(z_m))\ge\delta\cdot
p(a,b)
\end{aligned}
$$
for distinct $k,m\in\IN$. This implies that the sequence
$\{z_k\}_{k=1}^\infty$ is not $T_n(p)$-Cauchy. Since $T(p)\ge
2^{-n}T_n(p)$ and $\{z_k\}_{k=1}^\infty \subset \{y_k\}_{k=1}^\infty$ this
yields that the sequence $\{y_k\}_{k=1}^\infty$ is not $T(p)$-Cauchy. This
contradiction shows that the sequence $\{y_k\}_{k=1}^\infty$ is $d$-Cauchy
and consequently has a limit point $y_0\in Y$ (recall that the metric $d$
is complete).  
\end{proof}

\begin{lemma}\label{l:3.8} If the metric $d$ on $Y$ is totally bounded and 
$\dim Y\bs X<\infty$ then for every totally bounded pseudometric $p$ on $X$, 
$T(p)$ is a totally bounded pseudometric on $Y$.
\end{lemma}

\begin{proof} Fix a totally bounded pseudometric $p$ on $X$. Since $T(p)=
\sum_{n=1}^\infty T_n(p)$ and $\|T_n\|=1$, to show that the pseudometric $T(p)$
is totally bounded, it suffices to show that each $T_n(p)$ is a totally
bounded pseudometric. Fix any $n\in\IN$. Since the metric $d$ on $Y$ is totally 
bounded, by the construction, the cover $\U_n$ is finite. Then the metric $E_n(p)$
on $X\sqcup \U_n$ is totally bounded. Let $k\in\IN$ be such that $h(Y)\subset
SJ^k(X)$. Then $f_n(Y)\subset SJ(SJ^k(X\sqcup \U_n))$ and $T_n(p)(y,y')=
(sj\circ sj^k\circ E_n(p))(f_n(y),f_n(y'))$ for every $y,y'\in Y$. By 
Theorem~\ref{t:1.1}, the pseudometric $sj\circ sj^k\circ E_n(p)$ is totally bounded. Hence, 
the  pseudometric $T_n(p)$ is totally bounded as well. 
\end{proof}

It follows from Lemmas~\ref{l:3.1}--\ref{l:3.8} that the operator satisfies all the
requirements of Theorem~\ref{t:3.1}.

\begin{remark} Theorem~\ref{t:3.1} for compact metrizable $Y$ was proved earlier
by O.~Pikhurko \cite{14}.
\end{remark}


\section{Linear operator extending invariant metrics}

By $C(X\t X)$ we denote the linear lattice of continuous functions on $X\t
X$, equipped with the compact-open topology.
Let $G$ be a compact group and $X$ be a (left) $G$-space \cite{4}. Consider the
linear subspace $C_{inv}(X\t X)=\{f\in C(X\t X)\mid
f(gx,gy)=f(x,y)$ for every
$g\in G$ and $x,y\in X\}$ of $C(X\t X)$
consisting of continuous {\it invariant} functions.

\begin{theorem}\label{t:4.1} Let $G$ be a compact group. For every stratifiable
$G$-space $Y$ and a closed invariant subset $X\subset Y$, there exists a regular
continuous extension operator $T:C_{inv}(X\t X)\to C_{inv}(Y\t Y)$ 
 preserving the class of  invariant pseudometrics.
Moreover, if the group
$G$ is finite then the operator $T$ is continuous with respect to the 
pointwise convergence of functions.
\end{theorem}

\begin{proof} Let $Y$ be a stratifiable $G$-space and $X\subset Y$ be an invariant closed
subspace. Let $E:\IR^{X\t X}\to \IR^{Y\t Y}$ be the operator from Theorem~\ref{t:3.1}. Now for 
each $p\in C_{inv}(X\t X)$ let $T(p)\in C_{inv}(Y\t Y)$ be the
function defined by $T(p)(y,y')=\int_GE(p)(gy,gy')d\mu$ for $y,y'\in Y$, 
where $\mu$ is
the Haar measure for $G$ \cite[0.3.1]{4}. One can show that $T(p)$ is a continuous
invariant pseudometric for $Y$, provided $p$ is a continuous invariant
 pseudometric for $X$. Let $p:X\t X\to \IR$ be an invariant 
continuous function. Then
$T(p)(x,x')=\int_GE(p)(gx,gx')d\mu=\int_Gp(gx,gx')d\mu=\int_Gp(x,x')d\mu=
p(x,x')\int_Gd\mu=p(x,x'),\; (x,x')\in X\t X$, i.e. $T$ is an extension 
operator. Since $G\cdot K$ is a compact subset for every compactum
$K\subset X$ and the operator $E$ is continuous with respect to the
compact-open topology, so is the operator $T$.
Obviously that the operator $T$ is regular and it is continuous in the pointwise
topology, provided $G$ is a finite group.  
\end{proof}

\begin{theorem}\label{t:4.2} Let $G$ be a compact group, $Y$ be a metrizable
$G$-space and $X$ be a compact invariant non-degenerate set in $Y$.
For every $\e>0$ there exists a linear positive continuous extension operator
$T:C_{inv}(X\t X)\to C_{inv}(Y\t Y)$ with $\|T\|<1+\e$,
preserving the class of continuous invariant (pseudo)metrics as well as
the class of admissible invariant metrics.
Moreover, if the group $G$ is finite
then the operator $T$ is continuous with respect to the pointwise convergence
of functions.
\end{theorem}

\begin{proof} Let $E:C_{inv}(X\t X)\to C_{inv}(Y\t Y)$ be the operator from
Theorem~\ref{t:4.1},  $a,b\in X$ be two distinct point in $X$ and
$d$ be an admissible invariant metric for $Y$ with $\|d\|<\e$ (see \cite[p.69, ex.13]{7}).

 Let $\td:Y\t Y\to \IR$ be the
invariant pseudometric for $Y$ defined by
$$\td(y,y')=\min\{d(y,y'),d(y,X)+d(y',X)\}.$$ The formula
$T(p)=E(p)+p(a,b)\td$, $p\in C_{inv}(X\t X)$, defines a linear positive
extension operator $T:C_{inv}(X\t X)\to C_{inv}(Y\t Y)$ with $\|T\|<1+\e$
such that $T(p)$ is an invariant (pseudo)metric on $Y$, provided $p$ is an 
invariant continuous (pseudo)metric on $X$.

Now assume that $p$ is an invariant admissible metric on $X$. Let's show
that $T(p)$ is an admissible metric on $Y$. Fix $y\in Y$ and a
neighborhood $U\subset Y$ of $y$. Consider at first the case $y\in Y\bs
X$. Since the metric $d$ on $Y$ is admissible, $O_d(y,\delta)\subset U$
for some $0<\delta<d(y,X)$. We claim that $O_{T(p)}(y,\delta\cdot
p(a,b))\subset U$. Indeed, for every $y'\in O_{T(p)}(y,\delta\,p(a,b))$
the inequality $T(p)(y,y')=E(p)(y,y')+p(a,b)\cdot \tilde
d(y,y')<\delta\,p(a,b)$ implies $\tilde d(y,y')<\delta$. Writing $\tilde
d(y,y')=\min\{d(y,y'),d(y,X)+d(y',X)\}$ and recalling that $d(y,X)>\delta$
notice that $\tilde d(y,y')<\delta$ implies $d(y,y')<\delta$ and $y'\in
O_d(y,\delta)\subset U$.

Now assume that $y\in X$. Since the metric $p$ on $X$ is admissible there
exists $\e>0$ such that $\bar O_p(y,\e)=\{x\in X\mid p(x,y)\le
\e\}\subset U$. Denote $A=\bar O_{E(p)}(y,\e)=\{z\in Y\mid
E(p)(z,y)\le\e\}\subset U$. Obviously $A\cap X=\bar O_p(y,\e)\subset U$.
Since $X$ is compact there exists $\delta>0$ such that $A\cap
O_d(X,\delta)\subset U$, where $O_d(X,\delta)=\{z\in Y\mid
d(z,X)<\delta\}$ is the $\delta$-neighborhood of $X$ in $Y$. We claim that
$O_{T(p)}(y,\min\{\e,p(a,b)\delta\})\subset U$. Indeed, suppose that for
$z\in Y$  $T(p)(z,y)<\min\{\e,p(a,b)\delta\}$. Since
$T(p)(z,y)=E(p)(z,y)+p(a,b)d(z,X)$ this implies $E(p)(z,y)<\e$ and
$d(z,X)<\delta$. Consequently, $z\in A\cap O_d(X,\delta)\subset U$. Hence,
the metric $T(p)$ is admissible for $Y$.  
\end{proof}


\begin{thebibliography}{}

\bibitem{1} R.~Arens, {\it Extension of functions on fully normal spaces},
Pacific J. Math. {\bf 2} (1952), 11--22.

\bibitem{2} T.~Banakh, {\it AE(0)-spaces and regular operators extending 
(averaging) pseudometrics}, Bull. Polon. Acad. Sci. Ser. Sci. Math. 
{\bf 42} (1994), 197--206.

\bibitem{3} C.~Bessaga, {\it On linear operators and functors extending pseudometrics},
Fund. Math., {\bf 142} (1993), 101-122.

\bibitem{4} C.~Bessaga, {\it Functional analytic aspects of geomenty. Linear
extending of metrics and related problems}, Progress in Functional
Analysis, Proceedings of the Peniscola Meeting 1990 on the occasion of the 
60th birthday of Professor M.~Valdivia, North-Holland, Amsterdam 1992, 
P.247--257.

\bibitem{5} C.~Bessaga, A.~Pe\l czy\'nski, Selected topics in
infinite-dimensional topology, PWN, Warsaw, 1975.

\bibitem{6} C.J.R.~Borges, {\it On stratifiable spaces}, Pacific J. Math. 
{\bf 17} (1966), 1--16.

\bibitem{7} G.E.~Bredon, Introduction to compact transformation groups,
Academic Press, NY, London, 1972.

\bibitem{8} J.~Dugundji, {\it An extension of Tietze's theorem}, Pacific J. Math.
{\bf 1} (1951), 353--367.

\bibitem{9} R.~Engelking, General Topology, PWN, Warszawa, 1977.

\bibitem{10} V.V.~Fedorchuk, V.V.~Filippov. General Topology. Basic constructions,
Moscow Univ. Press., Moscow, 1988 (in Russian).

\bibitem{11} F.~Hausdorff, {\it Erweiterung einer stetigen Abbildung}, 
Fund. Math. {\bf 30} (1938), 40--47.

\bibitem{12} R.~Haydon, {\it On a problem of Pe\l czy\'nski: Milutin spaces, 
Dugundji spaces and AE(0-dim)}, Studia Math. {\bf 52} (1974), 23--31.

\bibitem{13} A.~Pe\l czy\'nski, {\it Linear extensions, linear averaging, 
and their applications to linear topological classification of spaces of
continuous functions}, Dissert. Math. {\bf 58} (1968), 1--89.

\bibitem{14} O.~Pikhurko, {\it Extending metrics in compact pairs}, 
Mat. Studi\"\i. {\bf 3} (1994), 103--106.
\end{thebibliography}
\end{document}